\documentclass[10pt]{amsart}
\usepackage{amsthm,amssymb,latexsym,amsmath}
   \usepackage[latin1]{inputenc}

\def \fol {{\mathcal F}}
\def \F {{\mathcal F}}

\newtheorem{proposition}{Proposition}[section]
\newtheorem{definition}[proposition]{Definition}

\newtheorem{lema}[proposition]{Lemma}
\newtheorem{teo}[proposition]{Theorem}

\begin{document}

\title[Classification of holomorphic   foliations  on Hopf manifolds]
{Classification of holomorphic foliations of arbitrary codimension on Hopf manifolds}

\author{Antonio Marcos  Ferreira da Silva}

\address{Antonio Marcos Ferreira da Silva\\
Departamento de Ciências Exatas\\
Universidade Federal de Lavras\\
 37200 000 Lavras MG, Brazil} \email{ antoniosilva@dex.ufla.br}


\begin{abstract}
We classify nonsingular holomorphic foliations of arbitrary codimension on
certain Hopf manifolds. We prove that all holomorphic distribution of codimension $k$ on a generic Hopf manifold is induced by a mononial 
holomorphic $k-$form.

\end{abstract}
\maketitle
%\tableofcontents
\section{Introduction}
Let $W=\mathbb{C}^n-\{0\}$, $ n\geq 2$, and $f(z_1, z_2,...,
z_n)=(\mu_1 z_1,\mu_2 z_2,...,\mu_n z_n)$ be a diagonal contraction 
in $\mathbb{C}^n$, where $0<|\mu_i|<1$ for
all $1\leq i\leq n$. The  quotient space
$X=W/<f>$ is a compact, complex manifold of dimension $n$ called of
Hopf manifold. When $\mu_1=\dots=\mu_{n}$ we say that $X$ is a classical Hopf manifold.
Classical Hopf manifolds were first studied by Heinz Hopf \cite{hopf} in 1948. Hopf showed that $X$ is difeomorphic to the product of odd spheres $S^1\times S^{2n-1}$ and has
a complex structure which is not Kälher. K. Kodaira\cite{ko} classified all Hopf surfaces, but the problem of classification of general Hopf 
manifolds it is an open problem. The   geometry and topology of  Hopf 
manifolds have been studied by several authors, see for instance, Dabrowski \cite{Da}, Haefliger \cite{Hae}, Ise \cite{Ise}, Wehler\cite{weh} etc.

We are interested on holomorphic foliations on Hopf manifolds. D. Mall in \cite{Ma1} using the Kodaira's classification of Hopf surfaces obtained the classification 
of nonsingular holomorphic foliations on Hopf surfaces. E. Ghys studied 
holomorphic foliations on homogeneous spaces, and as consequence obtained the classification of codimension one foliation on classical Hopf manifolds. In \cite{aca} and in the author's Ph.D thesis \cite{tesis} with the advisors A. Fern\'andez-P\'erez and 
M. Corr\^ea JR we classified holomorphic 
foliation of dimension and codimension one on certain Hopf manifolds. In this paper we want to complete the classification started in \cite{aca} and \cite{tesis},  therefore we are interested  on holomorphic foliations of codimension $k$ where $1<k<n-1$. We will consider  the following types of Hopf manifolds

\begin{definition} 
We say that
\begin{enumerate}
\item $X$ is\textbf{ classical } if $\mu=\mu_1=\ldots=\mu_n$.
\item $X$ is \textbf{generic} if  $0<|\mu_1|\leq|\mu_2|\leq\ldots\leq|\mu_n|<1$
 and there not exists non-trivial relation between the $\mu_i$'s in this way
$$\prod_{i\in A}\mu^{r_{i}}_{i}=\prod_{j\in B}\mu^{r_{j}}_{j},\quad r_i,r_j
\in\mathbb{N},\quad A\cap B=\emptyset,\quad A\cup B=\{1,2,\ldots,n\}.$$
\item $X$ is \textbf{intermediary} if $\mu_1=\mu_2=\ldots=\mu_r$, where 
$2\leq r\leq n-1$ and there not exists non-trivial relation between the $\mu_i$'s in this way
$$\prod_{i\in A}\mu^{r_{i}}_{i}=\prod_{j\in B}\mu^{r_{j}}_{j},
\quad r_i,r_j\in\mathbb{N},\quad A\cap B=\emptyset,\quad A\cup B=\{1,r+1,\ldots,n\}.$$
\end{enumerate}
\end{definition}
A line bundle $L$ on $X$ is the quotient of  $W\times \mathbb{C}$  by the
operation of a representation of the fundamental group
of  $X$,
$
  \varrho_{L}:   \pi_1(X)\simeq \mathbb{Z}  \longrightarrow   GL(1,\mathbb{C})  =  \mathbb{C}^*
$ in the following way
$$
\begin{array}{ccc}
  W\times \mathbb{C}& \longrightarrow &    W\times \mathbb{C}  \\
 (z,v)& \longmapsto  &     (f(z),  \varrho_{L}(1)v)\
\end{array}
$$

We write $L=L_b$ for the bundle induced by the representation $
\varrho_{L}(\gamma)$ with  $b=\varrho_{L}(1)$. Our main result is the following:
\begin{teo}\label{teo1}
Let $X$ be a Hopf manifold, $\dim X \geq 3$, and let
$\mathcal{F}$ be a nonsingular holomorphic distribution 
on $X$ of codimension $k$ ($1<k< n-1$) given
by a nonzero twisted
differential $k$-form $\omega\in H^0(X,\Omega^{k}_X\otimes
L_b)$ with coefficients in the line bundle
$L_b=\mathcal{N}:=\det(N_\F)$.

Then the following holds:
\begin{itemize}
\item[(i)] If $X$ is classical, then  $b =\mu^{m}$ with  $m 
\in \mathbb{N}$ and $m\geq
k$. Furthermore $\fol$ is induced by a polynomial $k$-form
$$\omega=\sum\limits _{i_1<\dots<i_k} g_{i_1\dots i_k} dz_{i_1,\dots ,i_k},$$ where $dz_{i_1,\dots ,i_k}$ denotes $dz_{i_1}\wedge \dots \wedge dz_{i_k}$ and $g_{i_1\dots i_k}$ are  homogeneous
polynomial of the same degree  $m-k$, such that
$\cap _{i_1<\dots<i_k} \{g_{i_1\dots i_k}=0\}=\{0\}$.

\item[(ii)]
 If $X$ is generic, then  $b =\mu_{i_1}\dots\mu_{i_k}$ for some $i_1<\dots <i_k$, 
and $\fol$ is induced by the $k$-form $\omega= dz_{i_1.\dots,i_k}$.

\item[(iii)] If $X$ is intermediary then $b=\mu^m\mu_{i_1}\dots \mu_{i_h}$ where $h\in \mathbb{N}\,, 0\leq h\leq k$, and $m+h=k$ or $m+h=k+1$. \\If $m+h=k$
then the foliation is induced by a holomorphic
 $k-$form of the type 
 \begin{center}
 $\omega=\sum\limits 
 _{j_{1}<\dots <j_{k-h}\leq r<i_1<\dots<i_h }  c^{j_{1}\dots j_{k-h}} 
 dz_{j_1,\dots,j_{k-h}, i_1,\dots,i_h}$ 
\end{center}
 If $m+h=k+1$ then distribution is induced by a linear holomorphic  $k-$ form.
\end{itemize}
\end{teo}

When X is a generic Hopf manifold, we have the following   result:

\begin{teo}\label{generic}
All holomorphic distribution of codimension $k$  (possibly singular) on a generic Hopf manifold of dimension at least  three are induced by a monomial $k$-form.

\end{teo}

%\begin{cor}
%All nonsingular  holomorphic distribution  of codimensin $k$ on a
%generic Hopf manifold are integrable.
%\end{cor}
%proof. This corollary follows directly from Theorem \ref{teo1} item $(ii)$ because the distribution \F is indueb bu an exact k-form.

\section{Holomorphic foliations}
Let $X$ be a complex manifold. A (nonsingular) \emph{foliation} $\F$, of
dimension $k$, on $X$   is a subvector bundle $T\F \hookrightarrow
T_X$, of generic rank $k$, such that $[T\F,T\F]\subset T\F$.

There is a dual point of view where $\mathcal F$ is determined by a
subvector bundle $N^*_{ \mathcal F}$, of   rank $n-k$, of the
cotangent bundle $\Omega^1_X = T^* X$ of $X$. The  vector bundle
$N_{\mathcal{F}}^*$ is called \emph{conormal vector bundle} of
$\fol$. The involutiveness   condition  is
replace by: if $d$ stands for the exterior derivative
then $dN_{\mathcal{F}}^* \subset N_{\mathcal{F}}^* \wedge
\Omega^1_X$ at the level of local sections.
The normal bundle $N_{ \mathcal F}$ of $\mathcal F$ is defined as
the dual of $N_{\mathcal{F}}^* $. We have the following exact
sequence
\[
0 \to T\mathcal F \to TX \to N_{\mathcal F} \to 0 \, .
\]

The $(n-k)$-th wedge product of the inclusion $N^*_\F
\hookrightarrow \Omega^1_X$ gives rise to a nonzero twisted
differential $(n-k)$-form $\omega\in H^0(X,\Omega^{n-k}_X\otimes
\mathcal{N})$ with coefficients in the line bundle
$\mathcal{N}:=\det(N_\F)$, which is \emph{locally decomposable} and
\emph{integrable}.
By construction the tangent bundle of a Hopf manifold $X$ is given
by
$$
TX=\bigoplus_{i=1}^n L_{\alpha_i},
$$
where $L_{\alpha_i}$ is the tangent bundle of the foliation induced
by the canonical vector field $\frac{\partial}{ \partial z_i}$.

\section{Cohomology of line bundles on Hopf manifolds}
\par Let $\Omega^{p}_{X}$ be the sheaf of germs of holomorphic $p$-forms
on a Hopf manifold $X$. Denote by $\Omega^{p}_{X}(L_b):=\Omega^{p}_{X}\otimes L_b$ 
and by $\pi:W\to X$ the natural projection on $X$. Consider an open covering $\{U_{i}\}$
 of $X$ such that all sets open $U_{i}$ are Stein, simply-connected  
and $\tilde{U}_{i}:=\pi^{-1}(U_{i})$ is a disjoint union of Stein open sets on $W$. 
Since $\pi$ is surjective, we have $A=\{\tilde{U}_{i}\}$ is open covering of $W$, therefore, the definition of $X$ implies that

$$\displaystyle\tilde{U}_{i}=\cup_{r\in\mathbb{Z}}f^{r}(U_{i}).$$

\par Let $\varphi\in\Gamma(U_i,\Omega^{p}_{X}(L_b))$.  Then $\tilde{\varphi}=
\pi^{*}(\varphi)$ belongs to $\Gamma(\tilde{U}_i,\pi^{*}(\Omega^{p}_{X}(L_b)))
\cong\Gamma(\tilde{U}_i,\Omega^{p}_{W})$. Therefore we have an exact sequence of C\u{e}ch complexes
\begin{equation}\label{sequencia de Cech}
0\rightarrow \mathcal{C}^\textbf{.}(A, \Omega_{X} ^p(L_b))\stackrel{\pi^*}
{\longrightarrow}
\mathcal{C}^\textbf{.}(A, \Omega^p _W) \stackrel{bId-f^*}{\longrightarrow} 
\mathcal{C}^\textbf{.}(A, \Omega_W ^p)\longrightarrow 0.
\end{equation}
From (\ref{sequencia de Cech}) we derive the long exact sequence of cohomology
\begin{eqnarray*}
&&0\rightarrow H^0(X, \Omega_{X} ^p (L_b))
\rightarrow H^0(W, \Omega_W ^p)
 \stackrel{p_0}{\rightarrow} H^0(W, \Omega_W ^p) \rightarrow  H^1(X, \Omega_{X} ^p (L_b))\rightarrow
\cdots \end{eqnarray*}		
where
$p_0=b\cdot Id-f^*: H^0 (W, \Omega_W ^p)\rightarrow H^0 (W, \Omega_W ^p)$
and $W=\mathbb{C}^{n}-\{0\}$.
D. Mall proved in  \cite{Ma} the following result.
\begin{teo}[Mall \cite{Ma}] \label{teocoh}
If $X$ is a Hopf manifold of dimension $n\geq 3$ and $L_b$ is a line bundle on $X$. Then
\begin{eqnarray*}\label{equacao99}
&&dim H^0(X, \Omega_{X} ^p (L_b))=dim\,Ker(p_0)
\end{eqnarray*}
\end{teo}

To prove the Theorem \ref{teo1} we need proved the following lemma.

\begin{lema}\label{le1}
Let $X$ be a classical, generic or intermediary Hopf manifold of dimension 
$n\geq3$, and $L_b$ be a line bundle on $X$, with $b\in \mathbb{C}^*$. The following holds:
 \begin{enumerate}
\item[(i)] If $X$ is classical then $\dim\,H^0(X,\Omega_X ^k \otimes L_b)>0$ 
if, and only if, $b=\mu^m$, where $m\in \mathbb{N}, \, m\geq k$.

\item[(ii)] If $X$ is generic then $\dim\, H^0(X,\Omega_X ^k \otimes L_b)>0$ 
if, and only if, $$b=\mu_1 ^{m_1}\mu_2 ^{m_2}\dots \mu_n ^{m_n}$$ where $m_i\in 
\mathbb{N}$ and there exists $j_1,\dots,j_k\in \{1,\dots  ,n\}$, such that \mbox{$m_{j_0},\dots,m_{j_k}\geq 1$.}

\item[(iii)] If $X$ is intermediary then  $\dim\,H^0(X,\Omega_X ^k \otimes L_b)>0$ 
if, and only if, $$b=\mu_1^{m_1} \mu_2^{m_2}\dots \mu_n^{m_n}$$ with
$\mu_j \in \mathbb{N}$ for all $j=1,\dots ,n$, and:
\begin{enumerate}
 \item[0)] $m_1+m_2+\dots +m_r= 0$, and there exists $i_1,\dots ,i_k\geq 
r+1$ such that \\ $m_{i_1}\geq 1 ,\dots,  m_{i_k}\geq 1$ , or 
\item[1)] $m_1+m_2+\dots +m_r= 1$, and there exists $i_1,\dots ,i_{k-1}\geq
 r+1$ such that \\$m_{i_1}\geq 1,\dots, m_{i_{k-1}}\geq 1$, or 
\\.\\.\\.
\item [k-1)] $m_1+m_2+\dots +m_r= k-1$, and there exists $i_1\geq r+1$ such 
that $m_{i_1}\geq 1$, or 
\item [k)] $m_1+m_2+\dots +m_r\geq k$.
\end{enumerate}
\end{enumerate}
\end{lema}
\begin{proof}
By Theorem \ref{teocoh} we have $\dim H^0(X, \Omega_X ^{k}\otimes L_b)=\dim(ker\, p_0)$, where
\begin{center}
$p_0: H^0 (W, \Omega_ W ^{k})\longrightarrow H^0 (W, \Omega_ W ^{k})$, \quad$p_0=b
 Id-f^*$ \, and \, $W=\mathbb{C}^n-\{0\}$.
\end{center}

Let $\omega \in H^0(W,\Omega_W ^{k})$, then
$\omega= \sum\limits _{i_1<\dots<i_k}  g_{i_1\dots i_k} dz_{i_1,\dots, i_k}$. 
It  follows from Hartogs extension theorem that each $g_{i_1\dots i_k}$ 
can be represented by its Taylor series
\begin{center}
$g_{i_1\dots i_k}(z_1,z_2,\dots ,z_n)=\sum\limits _{\alpha \in \mathbb{N}^n} 
c_\alpha ^{i_1\dots i_k} z_1^{\alpha_1}z_2^{\alpha_2}\dots z_n ^{\alpha_n}$, for all $i=1,\dots ,n.$
\end{center}
 
Hence
\begin{equation}\label{equacao123}
p_0(\omega)=\sum\limits _{i_1<\dots <i_k} \sum\limits _{\alpha \in \mathbb{N}^n}
c_\alpha ^{i_1\dots i_k} (b-\mu_1 ^{\alpha_1}\dots \mu_n ^{\alpha_n}\mu_{i_1}\dots 
\mu_{i_k}) z_1^{\alpha_1} \dots z_{n} ^{\alpha_n} dz_{i_1,\dots, i_k}.
\end{equation}

In the classical case,  $\mu_1=\dots =\mu_n=\mu$ and
\begin{equation*}p_0(\omega)=\sum\limits _{i_1<\dots <i_k} \sum\limits _
{\alpha \in \mathbb{N}^n}c_\alpha ^{i_1\dots i_k} (b-\mu^{\alpha_1+\dots +\alpha_k +k})
z_1^{\alpha_1} \dots z_{n} ^{\alpha_n}dz_{i_1,\dots, i_k}. \end{equation*}
thus $\dim(ker\,p_0) >0$ if, and only if,
$b=\mu^{m}$, for some $m \in \mathbb{N}$, $m\geq k$.

In the generic case, since $\mu_i's$ have no relations, it  follows from (\ref{equacao123}) 
that \\ \mbox{$\dim(ker\,p_0) >0$} if, and only if,  $b=\mu_1 ^{m_1} \mu_2^{m_2}\dots \mu_n ^{m_n}$ 
where $\mu_j \in \mathbb{N}$, and there exists $j_1,\dots,j_k\in \{1,\dots  ,n\}$, 
such that \mbox{$m_{j_1},\dots,m_{j_k}\geq 1$.}

Finally, for the intermediary case, we have $\mu_1=\dots =\mu_r=\mu$, so that
\begin{equation*}
 p_0(\omega)=\sum\limits _{i_1<\dots <i_k} \sum\limits _{\alpha \in 
\mathbb{N}^n}c_\alpha ^{i_1\dots i_k}
 (b-\mu^{\alpha_1+\dots+\alpha_k}\mu_{r+1}^{\alpha_{r+1}}\dots \mu_n ^{\alpha_n}\mu_{i_1}\dots
 \mu_{i_k}) z_1^{\alpha_1} \dots z_{n} ^{\alpha_n}dz_{i_1,\dots, i_k}.
\end{equation*}

Since $\mu, \mu_{r+1},\dots ,\mu_n$ have no relations, we have
$\dim(ker\,p_0)>0$ if, and only if, $b=\mu^m \mu_{r+1}^{m_{r+1}}\dots \mu_n^{m_n}$ such that
$\mu_j \in \mathbb{N}$ for all $j=1,\dots ,n$, and:
\begin{enumerate}
 \item[0)] $m_1+m_2+\dots +m_r= 0$, and there exists $i_1,\dots ,i_k\geq 
r+1$ such that \\ $m_{i_1}\geq 1 ,\dots,  m_{i_k}\geq 1$ , or 
\item[1)] $m_1+m_2+\dots +m_r= 1$, and there exists $i_1,\dots ,i_{k-1}
\geq r+1$ such that \\$m_{i_1}\geq 1,\dots, m_{i_{k-1}}\geq 1$, or 
\\.\\.\\.
\item [k-1)] $m_1+m_2+\dots +m_r= k-1$, and there exists $i_1\geq 
r+1$ such that $m_{i_1}\geq 1$, or 
\item [k)] $m_1+m_2+\dots +m_r\geq k$.
\end{enumerate}\end{proof}

\subsection{Proof  of Theorem \ref{teo1}}
\begin{proof}
By construction,  we have that
a holomorphic section $s \in H^0(X, \Omega_X ^k \otimes L_{b})$ corresponds to a section
 $\tilde{s} \in H^0(W, \mathcal{O}_W ^{n \choose k})$, 
say $\tilde{s}=(g_{i_1\dots i_k})_{i_1<\dots < i_k}$, such that
$g_{i_1\dots i_k} \in \mathcal{O}_W$ satisfies
$$g_{i_1\dots i_k}(\mu_1 z_1,\dots ,\mu_n z_n)= \mu_{i_1}^{-1}\dots \mu_{i_k} ^{-1} 
b  g_{i_1\dots i_k}(z_1,\dots ,z_n),$$ for all ${i_1<\dots < i_k}.$
By Hartog's extension theorem, $\tilde{s}$ can be represented by its Taylor series
\begin{center}
$g_{i_1\dots  i_k}(z_1,\dots ,z_n)=\sum\limits _{\alpha \in \mathbb{N}^n} c_\alpha 
^{i_1\dots  i_k} z_1^{\alpha_1}\dots  z_n ^{\alpha_n}$, where $\alpha=(\alpha_1,
\alpha_2,\dots ,\alpha_n) \in \mathbb{N}^n$ .
\end{center}
Then
\begin{equation}\label{equa6}
c_{\alpha}^{i_1\dots  i_k} \mu_1 ^{\alpha_1} \mu_2 ^{\alpha_2}\dots  \mu_ n ^{\alpha_n}=
c_{\alpha}^{i_1\dots  i_k} \mu_{i_1}^{-1}\dots \mu_{i_k} ^{-1} b,
 \end{equation}
$\textrm{ where } \alpha =(\alpha_1, \alpha_2, \dots ,\alpha_n)\in \mathbb{N}^n.$\\ \\
\textbf{Classical case.} In this case $\mu_1=\dots = \mu_n=\mu$. Lemma \ref{le1} part $(i)$ implies that
 $b=\mu^{m}$ for some
$m\geq k$.
Therefore
$$
c_{\alpha}^{i_1\dots i_k}\mu^{|\alpha|}=c_{\alpha}^{i_1\dots i_k} \mu^{-k}\mu^{m} \textrm{ where } 
|\alpha|=\alpha_1+\dots +\alpha_n.
$$
Hence, if $ c_{\alpha}^{i_1\dots i_k}\neq 0$ then $|\alpha|=m-k$. It follows that  each $g_{i_1\dots i_k}$ 
is a homogeneous polynomial of degree $m-k$.
\\ \\
\textbf{Generic case.} If $X$ is generic, then by Lemma \ref{le1} part $(ii)$ we have
$$b=\mu_1 ^{m_1} \mu_2^{m_2}\dots \mu_n ^{m_n}$$  where $m_i\in \mathbb{N}$ 
and there exists $j_1,\dots,j_k\in \{1,\dots  ,n\}$, 
such that \mbox{$m_{j_0},\dots,m_{j_k}\geq 1$.} Then from (\ref{equa6}) we get
\begin{center}
$ 
c_{\alpha}^{i_1\dots  i_k} \mu_1 ^{\alpha_1} \mu_2 ^{\alpha_2}\dots  \mu_ n ^{\alpha_n}
=c_{\alpha}^{i_1\dots  i_k} 
\mu_{i_1}^{-1}\dots \mu_{i_k} ^{-1}\mu_1 ^{m_1} \mu_2 ^{m_2}\dots \mu_n ^{m_n}$ 
\end{center}
where $\alpha =(\alpha_1, \alpha_2, \dots ,\alpha_n)\in \mathbb{N}^n$.

Hence for each ${i_1< \dots < i_k}$ we have
\begin{equation}\label{generica}
g_{i_1\dots  i_k}(z_1,\dots  ,z_n)=c_{\alpha_0} ^{{i_1\dots  i_k}} z_{1}^{m_1}z_2 ^{m_2}
\dots z_n^{m_n}z_{i_1} ^{-1}\dots z_{i_k} ^{-1} .
\end{equation}
Since $\fol$ is nonsingular, we get that $m_{i_1}=\dots m_{i_k}=1$ for some ${i_1< 
\dots < i_k}$ 
and $m_i=0$ for all $i\in\{1,\dots, n\}\setminus\{i_{1}\dots i_k\}$. \\

So that we have $b=\mu_{i_1}\dots\mu_{i_k}$, $g_{i_{1}\dots i_k}$ is a constant
and $g_{j_1\dots j_k}=0$ for all ${j_1<\dots <j_k}\neq i_{1}<\dots <i_k$.
\\ \\
\textbf{Intermediary case.}
In this case $\mu_1=\dots =\mu_r$, then by Lemma \ref{le1} part $(iii)$ we have 
$$b=\mu  ^{m } \mu_{r+1}^{m_{r+1}}\dots \mu_n ^{m_n}$$ 
 Then from (\ref{equa6}) we get
\begin{center}
$c_{\alpha}^{i_1\dots  i_k} \mu^{\alpha_1+\dots +\alpha_r}\mu_{r+1}^{\alpha_{r+1}} 
\dots  \mu_ n ^{\alpha_n}=c_{\alpha}^{i_1\dots  i_k} \mu_{i_1}^{-1}\dots \mu_{i_k} ^{-1} 
\mu  ^{m } \mu_{r+1}^{m_{r+1}}\dots \mu_n ^{m_n} $\end{center}
where $\alpha =(\alpha_1, \alpha_2, \dots ,\alpha_n)\in \mathbb{N}^n.$

Hence, if $c_{\alpha}^{i_1\dots  i_k}\neq 0$ then  
$$\mu^{\alpha_1+\dots +\alpha_r}
\mu_{r+1}^{\alpha_{r+1}} \dots  \mu_ n ^{\alpha_n}= \mu_{i_1}^{-1}\dots \mu_{i_k} ^{-1} \mu  
^{m } \mu_{r+1}^{m_{r+1}}\dots \mu_n ^{m_n}$$

where $\alpha =(\alpha_1, \alpha_2, \dots ,\alpha_n)\in \mathbb{N}^n.$\\
If $i_1,\dots ,i_k \leq r$, we have   $\alpha_1+\dots +\alpha_k=m-k$, $\alpha_j=m_j$ 
for $j\geq r+1$ and \\

$g_{i_1\dots i_k}=\sum\limits _{\alpha \in \mathbb{N}^r} c_{\alpha}^{i_1\dots  i_k} z_1 
^{\alpha_1}\dots z_r ^{\alpha_r}z_{r+1} ^{m_{r+1}}\dots z_n ^{m_n}$ with $|\alpha|=m-k$,
where $\alpha=(\alpha_1,...,\alpha_r)$ and $|\alpha|=\alpha_1+\dots +\alpha_r$.
\\
If $i_1,\dots ,i_{k-1} \leq r$ and $i_k \geq r+1$ we have  
\\

$g_{i_1\dots i_k}=\sum\limits _{\alpha \in \mathbb{N}^r} c_{\alpha}^{i_1\dots  i_k} z_1 ^{\alpha_1}\dots 
z_r ^{\alpha_r}z_{r+1} ^{m_{r+1}}\dots z_n ^{m_n}z_{i_k} ^{-1}$ with $|\alpha|=m-k+1$,
where $\alpha=(\alpha_1,...,\alpha_r)$ and $|\alpha|=\alpha_1+\dots +\alpha_r$.

\begin{center}
 .\\.\\.
\end{center}

If $i_1,\dots ,i_{s} \leq r$ and $i_{s+1},\dots, i_k \geq r+1$ we have  
\\

$g_{i_1\dots i_k}=\sum\limits _{\alpha \in \mathbb{N}^r} c_{\alpha}^{i_1\dots  i_k} z_1 ^{\alpha_1}\dots 
z_r ^{\alpha_r}z_{r+1} ^{m_{r+1}}\dots z_n ^{m_n}z_{i_{s+1}} ^{-1}\dots z_{i_k} ^{-1}$ with $|\alpha|=m-s$,
where $\alpha=(\alpha_1,...,\alpha_r)$ and $|\alpha|=\alpha_1+\dots +\alpha_r$.

\begin{center}
.\\.\\.
\end{center}

If   $i_{1},\dots, i_k \geq r+1$ we have  
$$g_{i_1\dots i_k}=\sum\limits _{\alpha \in \mathbb{N}^r} c_{\alpha}^{i_1\dots  i_k} z_1 ^{\alpha_1}\dots 
z_r ^{\alpha_r}z_{r+1} ^{m_{r+1}}\dots z_n ^{m_n}z_{i_{1}} ^{-1}\dots z_{i_k} ^{-1}$$ with $|\alpha|=m$,
where $\alpha=(\alpha_1,...,\alpha_r)$ and $|\alpha|=\alpha_1+\dots +\alpha_r$.
We will divide the proof of intermediary case in four subcases:
 \\
\\ \\ 
1) $m=0$

In this case, $g_{i_1\dots i_k}=0$ if some $i_j\leq r$, 
and if   $i_{1},\dots, i_k \geq r+1$ we have  
$$g_{i_1\dots i_k}= c^{i_1\dots  i_k} z_{r+1} ^{m_{r+1}}\dots 
z_n ^{m_n}z_{i_{1}} ^{-1}\dots z_{i_k} ^{-1}$$ where $c^{i_1\dots  
i_k}\in \mathbb{C}$ is a constant. Note that $k\leq n-r$.

Since, $\fol$ is nonsingular, we get that $m_{j_1}=\dots=m_{j_k}=1$ 
for some $r< j_1<\dots< j_k\leq n$, and $m_i=0$ if $i\neq j_l$ 
for all $l=1,\dots,k$. Thus, $b=\mu_{j_1}\dots\mu_{j_k}$, $g_{j_1\dots j_k}=constant\neq 0$  and $g_{i_1\dots i_k}=0$
if ${i_1<\dots<i_k}\neq {j_1< \dots<j_k}$.  In this case the distribution is induced by a holomorphic
$k-$form of the type $\omega=dz_{j_1}\wedge \dots \wedge dz_{j_k}$ \\ \\ \\
2) $m=1$

Note that $m-s\geq 0 \Leftrightarrow s=0 \mbox{ or } s=1$, so $g_{i_1\dots i_k}=0$ if $i_1, i_2 \leq r$.
\\ 

If $i_1 \leq r$ and $i_2,\dots  \ i_k\geq r+1$ we have $$g_{i_1\dots i_k}=  c^{i_1\dots  i_k}  z_{r+1} ^{m_{r+1}}\dots z_n ^{m_n}z_{i_{2}} ^{-1}\dots z_{i_k}
 ^{-1}$$ where $c^{i_1\dots  i_k}$ is a constant.
\\

If $i_1,\dots  \ i_k\geq r+1$ we have $$g_{i_1\dots i_k}=\sum\limits _{\alpha \in \mathbb{N}^r} c_{\alpha}^{i_1\dots  i_k} z_1 ^{\alpha_1}\dots 
z_r ^{\alpha_r}z_{r+1} ^{m_{r+1}}\dots z_n ^{m_n}z_{i_{1}} ^{-1}\dots z_{i_k} ^{-1}$$ with $|\alpha|=1$,
where $\alpha=(\alpha_1,...,\alpha_r)$ and $|\alpha|=\alpha_1+\dots +\alpha_r$.
\\ \\ \\ 2.1)  $m=1,\, r=n-1$. 

Since $k\geq 2$, in this case there is distribution only if $k=2$.\\
If $i_1\leq r=n-1$ and $n=i_2\geq r=n-1+1$ then 

\begin{center}
$g_{i_1n}= c^{i_1 n}  z_{n} ^{m_{n}-1}$, where $c^{i_1 n }$ is a constant,
and $g_{i_1 i_2}=0$ in  all  other cases. 
\end{center}
 
Since $\fol$ is nonsingular, we have $m_n=1$, $b=\mu \mu_n$, and 
the distribution is induced by $k-$form of the type $\sum\limits_{i_1} ^{n} c^{i_1n}dz_{i_1}\wedge dz
_n$.
\\ \\ \\  
2.2)  $m=1,\, r=n-2$. 
 
In this case, there is ditribution only if $k=2$ or $k=3$.\\
\begin{itemize} \item $k=3$. \\ If $i_1\leq n-2=r$ and $i_2=n-1, i_3=n\geq r+1=n-1$  then 
$$g_{i_1 n-1 n}=c^{i_1 n-1 n}z_{n-1} ^{m_{n-1} } z_{n} ^{m_n} z_{n-1} ^{-1}  z_{n} ^ {{-1} }$$ where $c^{i_1 n-1 n }$ is a constant,
and $g_{i_1 i_2 i_3 }=0$ in  all  other cases. Since $\fol$ is nonsingular, $m_{n-1}=m_n=1$, $b=\mu \mu_{n-1}{\mu_n}$, and the distribution is induced
by a $3-$form of the type $\omega=\sum\limits_{i_1=1} ^{n-2}  c^{i_1 n-1 n } dz_{i_1} \wedge dz_{n-1} \wedge dz_n$.

\item $k=2$. \\
If $i_1\leq n-2$ and $i_2 \geq r+1=n-1$, then 
$$g_{i_1 i_2}= c^{i_1 i_2} z_{n-1} ^{m_{n-1}} z_n ^{m_n} z_{i_2} ^{-1}$$
\\ If $i_1=n-1, i_2 =n\geq r+1=n-1$ then 
$$g_{n-1 n}=\sum\limits _{\alpha \in \mathbb{N}^{n-2}} c_{\alpha}^{n-1 n} z_1 ^{\alpha_1}\dots 
z_{n-2} ^{\alpha_{n-2}}z_{n-1} ^{m_{n-1}}z_n ^{m_n} z_{n-1} ^{-1}z_{n} ^{-1}$$
 with $|\alpha|=1$,
where $\alpha=(\alpha_1,...,\alpha_{n-2})$ and $|\alpha|=\alpha_1+\dots +\alpha_{n-2}$.\\
We have three possibilities: \begin{itemize} \item [a)] $m_{n-1}=m_n=1$. In this case, 
$g_{i_1 n-1}= c^{i_1 n-1}  z_n $, $g_{i_1 n}= c^{i_1 n}  z_{n-1} $ if $i_1\leq n-2$, and 
$$g_{n-1 n}=\sum\limits _{\alpha \in \mathbb{N}^{n-2}} c_{\alpha}^{n-1 n} z_1 ^{\alpha_1}\dots 
z_{n-2} ^{\alpha_{n-2}}$$ with $|\alpha|=1$,
where $\alpha=(\alpha_1,...,\alpha_{n-2})$ and $|\alpha|=\alpha_1+\dots +\alpha_{n-2}$.
\item [b)] $m_{n-1}=1, m_n=0$. In this case, $g_{i_1 n-1}=constant$ if $i_1\leq n-2$ and $g_{i_1 i_2}=0$ in all other cases.

\item [c)] $m_{n-1}=0, m_n=1$. In this case, $g_{i_1 n}=constant$ if $i_1\leq n-2$ and $g_{i_1 i_2}=0$ in all other cases.
\\
\end{itemize} 
\end{itemize}
2.3) $m=1,\,2 \leq r=n-t \leq n-3$. 

If $i_1 \leq r=n-t$ and $i_2,\dots  \ i_k\geq r+1=n-t+1$ we have
$$g_{i_1\dots i_k}=  c^{i_1\dots  i_k}  z_{r+1} ^{m_{r+1}}\dots z_n ^{m_n}z_{i_{2}} ^{-1}\dots z_{i_k}
 ^{-1}$$ where $c^{i_1\dots  i_k}$ is a constant.
\\
If $i_1,\dots  \ i_k\geq r+1=n-t+1$ we have 
$$g_{i_1\dots i_k}=\sum\limits _{\alpha \in \mathbb{N}^r} c_{\alpha}^{i_1\dots  i_k} z_1 ^{\alpha_1}\dots 
z_r ^{\alpha_r}z_{r+1} ^{m_{r+1}}\dots z_n ^{m_n}z_{i_{1}} ^{-1}\dots z_{i_k} ^{-1}$$ with $|\alpha|=1$,
where $\alpha=(\alpha_1,...,\alpha_r)$ and $|\alpha|=\alpha_1+\dots +\alpha_r$.

Note that there is distribution only if $k\leq t+1$.\\
\begin{itemize} \item [a)] $k=t+1$.\\
If $i_1\leq r=n-k+1$ and $i_2,\dots, i_k \geq r+1=n-k+2$ then\\ $i_2=n-k+2, \dots, i_k=n$ and
$$g_{i_1, n-k+2, \dots, n}= c^{i_1, n-k+2 ,\dots, n} z_{n-k+2} ^{m_{n-k+2}} \dots z_n ^{m_n} z_{n-k+2} ^{-1}\dots z_{n}
 ^{-1}$$ where $ c^{i_1, n-k+2 ,\dots, n} $ is a constant, and $g_{i_1\dots i_k}=0$ in all other cases.

Since $\fol$ is nonsingular, $m_{n-k+2}=\cdots =m_n=1$, 
$b=\mu \mu_{n-k+2}\dots \mu_{n}$ and the distribution is induced by a constant $k-$form of the type $\omega=\sum \limits c^{i_1}dz_{i_1}\wedge dz_{n-k+2} \wedge \dots \wedge dz_{n}$,
where $c^{i_1}\in \mathbb{C}$ is a constant.
\item $k=t$\\

If $i_1\leq r=n-k$ and $i_2,\dots,i_k \geq r+1=n-k+1$ we have\\

$g_{i_1\dots i_k}=c ^{i_1\dots  i_k} z_{n-k+1} ^{m_{n-k+1}}\dots z_n ^{m_n}z_{i_{2}} ^{-1}\dots z_{i_k} ^{-1}$  where $c ^{i_1\dots  i_k} $ is a constant.
\\

If  $i_1,\dots,i_k \geq r+1=n-k+1$ we have $i_1=n-k+1 ,\dots,i_k=n $ and\\

$g_{i_1\dots i_k}=\sum\limits _{\alpha \in \mathbb{N} ^r} c_{\alpha} ^{i_1\dots  i_k} z_1 ^{\alpha_1}\dots 
z_r ^{\alpha_r}  z_{n-k+1} ^{m_{n-k+1}}\dots z_n ^{m_n}z_{n-k+1} ^{-1}\dots z_{n} ^{-1}$ with $|\alpha|=1$,
where $\alpha=(\alpha_1,...,\alpha_r)$ and $|\alpha|=\alpha_1+\dots +\alpha_r$.
We will separate into cases:
\begin{itemize}
\item [i)]$m_{n-k+1}=\dots=m_{n}=1$\\
In this case we have 

If $i_1\leq r=n-k$ and $i_2,\dots,i_k \geq r+1=n-k+1$ we have\\

$g_{i_1\dots i_k}=c ^{l}  z_{l}$  where $\{l\}=\{n-k+1,\dots ,n\}\setminus \{i_2,\dots i_k\}$ and  $c ^{l} $ is a constant.
\\

If  $i_1,\dots,i_k \geq r+1=n-k+1$ we have $i_1=n-k+1 ,\dots,i_k=n $ and\\

$g_{i_1\dots i_k}=\sum\limits _{i=1}  ^{ r} c ^i z_i $, where $c^i$ is constant for all $i$.\\

In this case, then we have linear distribution.

\item [ii)] If $m_{i}=1$ for all $n-k+1\leq i\leq n, i\neq n-k+l$ and $m_{n-k+l}=0$ we have \\

$g_{i_1 n-k+1\dots n-k+l-1, n-k+l+1 \dots i_n}=constant$ if $i_1\leq n-k$, and $g_{i_1\dots i_k}=0$ in all other cases. \\Then $b=\mu \mu_{n-k+1} \dots \mu_{n-k+l-1}\mu_{n-k+l+1} \dots \mu_{n}$, and 
the distribuition is induced by a $k-$form of the type \\$\omega=\sum\limits _{i =1} ^{n-k} c^{i } dz_{i }\wedge dz_{n-k+1} \wedge \dots dz_{n-k+l-1}\wedge dz_{n-k+l+1}\wedge\dots \wedge dz_{n}$ 
, where $c^i$ is constant for all $i$.

\end{itemize}
\item ${k < t}$\\
In this case,  there are $m_{j_1}=\dots m_{j_{k-1}}=1$, and $m_{j}=0$ for all $j\geq r+1, j\neq j_s, $ for all $s=1,\dots k-1$, $b=\mu \mu_{j_1}\dots mu_{j_{k-1}}$ and the distribution is induced by
a $k-$form of the type 
\\$\omega=\sum\limits _{i =1} ^{n-t} c^{i } dz_{i }\wedge dz_{j_1} \wedge  \dots \wedge dz_{j_{k-1}}$ 
, where $c^i$ is constant for all $i$.

\end{itemize}
Then the case $m=1$ is finished. It is easy see that if $k< m$ then there is no distribuition.\\
\vspace{0.5cm}
\\
3) $k=m>1$\\
Recall that: If $i_1,\dots ,i_{s} \leq r$ and $i_{s+1},\dots, i_k \geq r+1$,  where $s\leq k$, we have  
\\
$g_{i_1\dots i_k}=\sum\limits _{\alpha \in \mathbb{N}^r} c_{\alpha}^{i_1\dots  i_k} z_1 ^{\alpha_1}\dots 
z_r ^{\alpha_r}z_{r+1} ^{m_{r+1}}\dots z_n ^{m_n}z_{i_{s+1}} ^{-1}\dots z_{i_k} ^{-1}$ with $|\alpha|=m-s$,
where $\alpha=(\alpha_1,...,\alpha_r)$ and $|\alpha|=\alpha_1+\dots +\alpha_r$.\\
 
Suppose that $r<n-1$. In this case, since $\fol$ is nonsingular $m_l=0$ for all $l\geq r+1$, $b=\mu^m$ and the distribution is induced by a $k-$form of the type 
$\omega=\sum\limits _{i_1<\dots <i_k \leq r} c^{i_1 \dots i_k} dz_{i_1}\wedge \dots \wedge dz_{i_k}$.\\
Now, if $r=n-1$ then:\\
If $i_1,\dots, i_k\leq r=n-1$ 

$g_{i_1\dots i_k}= c^{i_1\dots  i_k}   z_{n} ^{m_n}$  
 where $c^ {i_1\dots  i_k}$ is a constant.\\
If $i_1,\dots, i_{k-1} \leq r=n-1$, $i_k=n$ \\
$g_{i_1\dots i_k}=z_n ^{m_n-1}\sum\limits _{i=1} ^{n-1} c ^{i} z_i $, where $c ^{i}$ is a constant for all $i$.\\
If $m_n=0$, then $g_{i_1\dots i_k}=constant$  for $i_1,\dots, i_k\leq r=n-1$, and $g_{i_1\dots i_k}=0$ in all other cases. Then the distribution is induced by a $k-$form of the type 
$\omega=\sum\limits _{i_1<\dots <i_k \leq r} c^{i_1 \dots i_k} dz_{i_1}\wedge \dots \wedge dz_{i_k}$. \\
If $m_n=1$, then $g_{i_1\dots i_k}=c^{i_1\dots  i_k}   z_{n}$ for $i_1,\dots, i_k\leq r=n-1$, and for  $i_1,\dots, i_{k-1} \leq r=n-1$, $i_k=n$ then 
$g_{i_1\dots i_k}= \sum\limits _{i=1} ^{n-1} c ^{i} z_i $, where $c ^{i}$ is a constant for all $i$.Then the distribution is induced by a linear $k-$form. 
\\ \\ \\
4)  $2 \leq m \leq k-1$\\ 
Suppose $m=k-\lambda, \, \lambda \in \mathbb{N},\, 1\leq \lambda \leq k-2$\\
Since $|\alpha|=m-s=k-\lambda-s$ then $|\alpha|\geq 0  \Leftrightarrow  s\leq k-\lambda$. Hence $g_{i_1\dots i_k}=0$ if $i_1,\dots , i_s \leq r$  and $s> k-\lambda$.\\
If $i_1,\dots , i_{k-\lambda-\gamma} \leq r $, $i_{k-\lambda-\gamma+1},\dots , i_k > r$, where $0	 \leq\gamma \leq m$ then 
\\

$g_{i_1\dots i_k}=\sum\limits _{\alpha \in \mathbb{N}^r} c_{\alpha}^{i_1\dots  i_k} z_1 ^{\alpha_1}\dots 
z_r ^{\alpha_r}z_{r+1} ^{m_{r+1}}\dots z_n ^{m_n}z_{i_{k-\lambda-\gamma+1}} ^{-1}\dots z_{i_k} ^{-1}$ with $|\alpha|=\gamma$,
where $\alpha=(\alpha_1,...,\alpha_r)$ and $|\alpha|=\alpha_1+\dots +\alpha_r$.\\
We will to consider $r=n-t$ where $1\leq t \leq n-2$ into separate cases .\\
\\ 
 4.1) $t=1$ \\
In this case, $i_{k-\lambda-\gamma+1},\dots , i_k >r=n-1 \Leftrightarrow \lambda=1, \gamma=0 \Rightarrow k=m+1$. Thus \\

$g_{i_1\dots i_k}=c^{i_1\dots i_k}z_{n} ^{m_n-1}$ if $i_1,\dots , i_{k-1} \leq r=n-1 $, $n=i_{k}> r=n-1$ and $g_{i_1\dots i_k}=0$ in 
all other cases. Since $\fol$ is nonsingular, we have $m_n=1$, $b=\mu^m \mu_{n}$, and the distribution is induced by
a $k-$form of the type 
\\$\omega=\sum\limits _{i_1<\dots < i_{k-1} \leq  n-1 }   c^{i_1,\dots , i_{k-1} }  dz_{i_1} \wedge  \dots \wedge dz_{i_{k-1}}\wedge d z_n$ 
, where $  c^{i_1,\dots , i_{k-1} }$ is a constant..
\\
4.2)$t=2$ \\
It is clear see that, in this case, does not have distribution if $\lambda \geq 3$.\\
If $\lambda=2$ we have $\gamma=0$, thus for $i_1, \dots , i_{k-2} \leq n-2, \, i_{k-1}=n-1, i_k=n >n-2$ we have  $g_{i_1\dots i_k}= c^{i_1 \dots i_{k-2}} z_{n-1} ^{m_{n-1}-1}z_{n} ^{m_{n}-1}$, where $ c^{i_1 \dots i_{k-2}}$ is a constant, and  $g_{i_1\dots i_k}=0$ in all other cases. Since $\fol$ is nonsingular, we have $m_{n-1}=m_{n}=1$, $b=\mu^m\mu_{n-1}\mu_{n}$, the distribution is induced by
a $k-$form of the type 
\\$\omega=\sum\limits _{i_1<\dots < i_{k-2} \leq  n-2 }   c^{i_1,\dots , i_{k-2} }  dz_{i_1} \wedge  \dots \wedge dz_{i_{k-2}}\wedge dz_{n-1}\wedge d z_n$ 
, where $  c^{i_1,\dots , i_{k-2} }$ is a constant.\\

If $\lambda=1$ we get $0\leq \gamma\leq 1$, thus for $i_1, \dots , i_{k-1} \leq r=n-2, \, i_k>r=n-2$ we have  $g_{i_1\dots i_k}= c^{i_1 \dots i_{k-1}} z_{n-1} ^{m_{n-1}}z_{n} ^{m_{n}}z_{i_k} ^{-1}$
and for $i_1, \dots , i_{k-2} \leq r=n-2, \, i_{k-1}=n-1,\, i_k=n>r=n-2$ we have  $g_{i_1\dots i_k}=\sum\limits _{\alpha \in \mathbb{N}^r} c_{\alpha}^{i_1 \dots i_{k}}z_1 ^{\alpha_1}\dots 
z_r ^{\alpha_r} z_{n-1} ^{m_{n-1}-1}z_{n} ^{m_{n}-1}$, with $|\alpha|=1$, and  $g_{i_1\dots i_k}=0$ in all other cases.\\
\begin{itemize}
\item If $m_{n-1}=m_n=1$. In this case  for $i_1, \dots , i_{k-1} \leq r=n-2, \, i_k=n-1>r=n-2$ we have $g_{i_1\dots i_{k-1} n-1}= c^{i_1 \dots i_{k-1}} z_{n} $,
for $i_1, \dots , i_{k-1} \leq r=n-2, \, i_k=n>r=n-2$ we have $g_{i_1\dots i_{k-1} n}= c^{i_1 \dots i_{k-1}} z_{n-1} $, and 
for $i_1, \dots , i_{k-2} \leq r=n-2, \, i_{k-1}=n-1,\, i_k=n>r=n-2$ we have  $g_{i_1\dots i_{k-2} n-1, n}=\sum\limits _{\alpha \in \mathbb{N}^r} c_{\alpha}^{i_1 \dots i_{k}}z_1 ^{\alpha_1}\dots 
z_r ^{\alpha_r}$, with $|\alpha|=1$. Thus the distribution is linear.
 
\item If $m_{n-1}=1, \, m_n=0$, then $g_{i_1\dots i_{k-1} n-1}=constant$ if $i_1,\dots i_{k-1}\leq n-2$ and $g_{i_1\dots i_{k}}=0$ in all other cases.

So the  distribution is induced by
a $k-$form of the type 
\\$\omega=\sum\limits _{i_1<\dots < i_{k-1} \leq  n-2 }   c^{i_1,\dots , i_{k-1} }  dz_{i_1} \wedge  \dots \wedge dz_{i_{k-1}}\wedge dz_{n-1} $ 
, where $  c^{i_1,\dots , i_{k-1} }$ is a constant.

\item If $m_{n-1}=0, \, m_n=1$, then $g_{i_1\dots i_{k-1} n}=constant$ if $i_1,\dots i_{k-1}\leq n-2$ and $g_{i_1\dots i_{k}}=0$ in all other cases.

So the  distribution is induced by
a $k-$form of the type 
\\$\omega=\sum\limits _{i_1<\dots < i_{k-1} \leq  n-2 }   c^{i_1,\dots , i_{k-1} }  dz_{i_1} \wedge  \dots \wedge dz_{i_{k-1}}\wedge dz_{n} $ 
, where $  c^{i_1,\dots , i_{k-1} }$ is a constant.

 \end{itemize}
4.3) $r=n-t, \, 2<t<n-1$.

Recall that $m=k-\lambda$ where $1\leq \lambda \leq k-2$, and that $g_{i_1\dots i_k}=0$ if $i_1,\dots, i_s\leq r=n-t$ and $s>k-\lambda$.\\
It is clear see that if  $\lambda > t$ there is no distribution.\\
\begin{itemize} 
\item [i)] Suppose $\lambda=t$, so for $i_1, \dots, i_{k-\lambda}\leq n-t$,
$i_{k-\lambda+1}=n-t+1, \dots ,i_k=n >n-t$ we have \\
$
g_{i_1\dots i_{k-\lambda},n-t+1\dots n}= c^{i_1\dots i_{k-\lambda},n-t+1\dots n} z_{n-t+1} ^{m_{n-t+1}-1}\dots z_{n} ^{m_{n}-1} $, and 
$g_{i_1\dots i_{k}}=0$ in all other cases. Since the distribution is nosingular we have $m_{n-t+1}=\dots=m_n=1$, $b=\mu^m \mu_{n-t+1}\dots \mu_n$, 
$ g_{i_1\dots i_{k-\lambda},n-t+1\dots n}= constant$ for  $i_1, \dots, i_{k-\lambda}\leq n-t$,  the  distribution is induced by
a $k-$form of the type \\

\begin{center}
$\omega=\sum\limits _{i_1<\dots < i_{k-\lambda} \leq  n-t }   c^{i_1\dots i_{k-\lambda},n-t+1\dots n }  dz_{i_1} \wedge  \dots \wedge dz_{i_{k-\lambda}}\wedge dz_{n-t+1}\wedge\dots \wedge dz_n $ 
, where $  c^{i_1\dots i_{k-\lambda}},n-t+1\dots n$ is a constant.
\end{center}

\item [ii)] Suppose now that $\lambda=t-a$, where $a\in \mathbb{N}$, $1\leq a \leq t-1$. In this case, we get \\
For each $0\leq \gamma \leq min\{a, m\}$, if $i_1,\dots \ i_{k-\lambda-\gamma} \leq r$, $i_{k-\lambda-\gamma+1},\dots ,i_k >r$ we have 
$g_{i_1\dots i_k}=\sum\limits _{\alpha \in \mathbb{N}^r} c_{\alpha}^{i_1\dots  i_k} z_1 ^{\alpha_1}\dots 
z_r ^{\alpha_r}z_{r+1} ^{m_{r+1}}\dots z_n ^{m_n}z_{i_{k-\lambda-\gamma+1}} ^{-1}\dots z_{i_k} ^{-1}$ with $|\alpha|=\gamma$,
where $\alpha=(\alpha_1,...,\alpha_r)$ and $|\alpha|=\alpha_1+\dots +\alpha_r$, and 
$g_{i_1\dots i_{k}}=0$ in all other cases.

\begin{itemize} 
\item  [a)] $\lambda\leq t-2$. Since $\fol$ is nonsingular then     there are $j_1,\dots ,j_{\lambda}>r$ such that $m_{j_1}=\dots =m_{j_{\lambda}}=1, m_i=0 $ if $i>r, i\neq j_l \forall l=1,\dots,\lambda$,
 $g_{i_1\dots i_{k-\lambda}j_1\dots j_{\lambda}}=constant$ for $i_1,\dots \ i_{k-\lambda } \leq r$
and $g_{i_1\dots i_{k}}=0$ in all others cases. Thus the  distribution is induced by
a $k-$form of the type \\

\begin{center}
$\omega=\sum\limits _{i_1<\dots < i_{k-\lambda} \leq  r }   c^{i_1\dots i_{k-\lambda},j_1\dots j_{\lambda} }  dz_{i_1} \wedge  \dots \wedge dz_{i_{k-\lambda}}\wedge dz_{j_1}\wedge\dots \wedge dz_{j_{\lambda}}$ 
, where $  c^{i_1\dots i_{k-\lambda}j_1\dots j_{\lambda}}$ is a constant.
\end{center}

\item [b)] $\lambda=t-1$, then $a=1$ and $min\{a,m\}=1$ .\\

For $i_1,\dots \ i_{k-\lambda } \leq r$, $i_{k-\lambda+1},\dots ,i_k >r$ we have \\
$g_{i_1\dots i_k}= c ^{i_1\dots  i_k}  z_{r+1} ^{m_{r+1}}\dots z_n ^{m_n} z_{i_{k-\lambda +1}} ^{-1}\dots z_{i_k} ^{-1}$,\\

For $i_1,\dots \ i_{k-\lambda-1 } \leq r$, $i_{k-\lambda},\dots ,i_k >r$ we have \\
 $g_{i_1\dots i_k}=\sum\limits _{\alpha \in \mathbb{N}^r} c_{\alpha}^{i_1\dots  i_k} z_1 ^{\alpha_1}\dots 
z_r ^{\alpha_r}z_{r+1} ^{m_{r+1}-1}\dots z_n ^{m_n-1} $ with $|\alpha|=1$,
where $\alpha=(\alpha_1,...,\alpha_r)$ and $|\alpha|=\alpha_1+\dots +\alpha_r$. 
\begin{itemize}
\item [i)] $m_{r+1}=\dots=m_{n}=1$. In this case, for $i_1,\dots \ i_{k-\lambda } \leq r$, $\{i_{k-\lambda+1},\dots ,i_k\}=\{r+1,\dots, n\}\setminus \{l\} $ we have
$$g_{i_1\dots i_k}= c ^{i_1\dots  i_k}  z_l,$$ where $c ^{i_1\dots  i_k}$ is a constant.

For $i_1,\dots \ i_{k-\lambda-1 } \leq r$, $\{i_{k-\lambda},\dots ,i_k\}=\{r+1,\dots,n\}$ we have
 $$g_{i_1\dots i_k}=\sum\limits _{\alpha \in \mathbb{N}^r} c_{\alpha}^{i_1\dots  i_k} z_1 ^{\alpha_1}\dots 
z_r ^{\alpha_r} $$ with $|\alpha|=1$,
where $\alpha=(\alpha_1,...,\alpha_r)$ and $|\alpha|=\alpha_1+\dots +\alpha_r$. Thus the distribution is linear.

\item [ii)] $m_h=0$ for some $h\geq r+1$, and $m_j=1 \, \forall  j\geq r+1, j\neq h$. In this case for 
$i_1,\dots \ i_{k-\lambda } \leq r$, $\{i_{k-\lambda+1},\dots ,i_k\}=\{r+1,\dots, n\}\setminus \{h\} $ we have
 $g_{i_1\dots i_k}= constant$, and $g_{i_1\dots i_{k}}=0$ in all others cases. 
Thus the  distribution is induced by
a $k-$form of the type 
 
$$\omega=\sum\limits c^{i_1\dots i_{k}}  dz_{i_1} \wedge  \dots \wedge dz_{i_k},$$
where $c^{i_1\dots i_{k}}$  is a constant,  $i_1,\dots \ i_{k-\lambda } \leq r$ and  $\{i_{k-\lambda+1},\dots ,i_k\}=\{r+1,\dots, n\}\setminus \{h\} $.

\end{itemize}

\end{itemize}
\end{itemize}
\end{proof}

\section{Proof  of Theorem \ref{generic}}
Recall   the Theorem \ref{generic}\\
\textbf{Theorem 1.3}
\textit{All holomorphic distribution of codimension $k$  (possibly singular) on a generic Hopf manifold of dimension at least  three are induced by a monomial $k$-form. }
\begin{proof} In fact,  follows directly from  proof of Theorem \ref{teo1} equation (\ref{generica}) that the distribution is induced by a monomial
$k-$form of the type $$\sum\limits _{i_1<\dots <i_k} g_{i_1\dots i_k}dz_{i_1\dots i_k}$$ where $g_{i_1\dots i_k}$ is a monomial .

\end{proof}
 
\vspace{3cm}

\noindent
\textbf{Acknowlegments.} The author is grateful to Maur\'icio Corr\^ea  and Arturo Fernández Pérez by suggest the theme and also
by the interesting conversations.


\begin{thebibliography}{12}
%\bibitem{atiyah}
%M. F. Atiyah and R. Bott : \emph{A Lefschetz Fixed Point Formula for Elliptic Complexes:
%II. Applications}. Annals of Mathematics, Second Series, Vol. 88, No. 3 (Nov., 1968), 451-491.
%\bibitem{Baum}
%P. Baum and R. Bott: \emph{Singularities of holomorphic foliations.} J. Differential Geom. 7 (1972), 279-432.
%\bibitem{cerveau}
%D. Cerveau: \emph{Pinceuax lin\'eaires de feuilletages sur $\mathbb{CP}(3)$ et conjecture de Brunella.} Publ. Mat. 46 (2002), 441-451.

\bibitem{aca}
M. Corr\^ea, A. Fern\'andez-P\'erez, A. M. Ferreira:\emph{ Classification of holomorphic foliations on Hopf manifolds.} Math.  Ann. Springer-Verlag Berlin Heidelberg (2015). DOI 10.1007/s00208-015-1267-z


\bibitem{Da}
K. Dabrowski:\emph{ Moduli Spaces for Hopf Surfaces}. Math. Ann. 259 (1982) 201-225.


\bibitem{tesis} 
A. M. Ferreira: \emph{Classificação de folheações holomorfas em variedades de Hopf.} Ph.D. thesis. Universidade Federal de Minas Gerais. Belo Horizonte (2014), 60p . Available online
at http://www.mat.ufmg.br/site/.


%SOBRENOME, PRENOME abreviado. Título: subtítulo (se houver). Data de defesa. Total de folhas. Tese (Doutorado) ou Dissertação (Mestrado) - Instituição onde a Tese ou Dissertação foi defendida. Local e data de   defesa. Descrição física do suporte


\bibitem{Ghys}
E. Ghys: \emph{Feuilletages holomorphes de codimension un sur les espaces
homogenes complexes.} Ann. Fac. Sci. Toulouse Math. (6) 5 (1996), no.
3, 493-519.
\bibitem{Hae}
A. Haefliger: \emph{Deformations of transversely holomorphic
flows on spheres and deformations of Hopf manifolds.}Compositio Mathematica 55 (1985), 241-251.

\bibitem{hopf}
H. Hopf:
\emph{ Zur Topologie der komplexen Mannigfaltigkeiten, Courant Birthday Volume, pp. 167-185, New York 1948.}


\bibitem{Ise}
M. Ise: \emph{On the geometry of Hopf manifolds.} Osaka
Math. J. 12 (1960), 387-402.


%\bibitem{izawa}
%T. Izawa: \emph{Residues of codimension one singular holomorphic distributions.} Bull. Braz. Math. Soc. New series 39 (3) (2008), 401-416.
\bibitem{Ma}
D. Mall: \emph{The cohomology of line bundles on hopf manifolds.} Osaka J Math.
28 (1991), 999-1015.
\bibitem{ko}
K. Kodaira:\emph{On the Structure of Compact Complex Analytic Surfaces, II. American Journal of Mathematics, Vol. 88, No. 3 (Jul., 1966), pp. 682-721.}


\bibitem{Ma1}
D. Mall: \emph{On holomorphic and transversely homomorphic foliations on Hopf surfaces.} J. reine angew. Math 501 (1998), 41-69.
%\bibitem{book}
%A. Lins Neto and B. Sc\'ardua: Introdu\c{c}\~ao \`a Teoria das Folhea\c{c}\~oes Alg\'ebricas Complexas. Edi\c{c}\~oes do 21 CBM, IMPA 1997. Available online at IMPA's website

\bibitem{weh}
J. Wehler:\emph{Versal deformation of Hopf surfaces.} J. reine angew. Math.328, 1981, 22-32.






\end{thebibliography}
\end{document}